\documentclass[11pt]{amsart}
\usepackage[normalem]{ulem}
\PassOptionsToPackage{table}{xcolor}
\usepackage{stmaryrd}
\expandafter\def\csname opt@stmaryrd.sty\endcsname
{only,shortleftarrow,shortrightarrow}
\usepackage{extpfeil}
\usepackage{amsmath,amssymb,amsthm}
\usepackage{aliascnt}
\usepackage{varioref}
\usepackage{hyperref}
\usepackage[capitalize,nameinlink,noabbrev]{cleveref}
\hypersetup{colorlinks=true,urlcolor=blue,citecolor=blue,linkcolor=blue}
\usepackage{courier}
\usepackage{tikz}
\usepackage{tikz-cd}
\usetikzlibrary{calc,matrix,arrows,decorations.markings}
\usepackage{array}
\usepackage{color}
\usepackage{enumerate}
\usepackage{nicefrac}
\usepackage{listings}
\usepackage{siunitx}
\usepackage{seqsplit}
\usepackage{longtable}
\usepackage{colonequals}
\usepackage{array, multirow}
\newcommand{\Q}{\mathbb{Q}}
\newcommand{\Qbar}{\overline{\Q}}
\newcommand{\R}{\mathbb{R}}
\newcommand{\C}{\mathbb{C}}

\newcommand{\Z}{\mathbb{Z}}

\newcommand{\Aut}{{\rm Aut}}

\renewcommand{\P}{\mathbb{P}}

\newcommand{\Jac}{\operatorname{Jac}}
\newcommand{\gon}{\operatorname{gon}}
\newcommand{\new}{\operatorname{new}}

\newcommand{\airr}{\operatorname{a.irr}}

\DeclareFontFamily{U}{wncy}{}
    \DeclareFontShape{U}{wncy}{m}{n}{<->wncyr10}{}
    \DeclareSymbolFont{mcy}{U}{wncy}{m}{n}
    \DeclareMathSymbol{\Sh}{\mathord}{mcy}{"58} 

\newtheorem{theorem}{Theorem}[section]

\newtheorem{lemma}[theorem]{Lemma}
\newtheorem{proposition}[theorem]{Proposition}
\newtheorem{corollary}[theorem]{Corollary}

\theoremstyle{definition}
\newtheorem{definition}[theorem]{Definition}

\newtheorem{remark}[theorem]{Remark}

\usepackage{xcolor}

\newcommand{\freddy}[1]{{\color{magenta} \textsf{$\spadesuit\spadesuit\spadesuit$ Freddy: [#1]}}}

\title{Bielliptic Shimura curves $X_0^D(N)$ with nontrivial level}

\begin{document}

\author{\sc Oana Padurariu}
\address{Oana Padurariu \\
Max-Planck-Institut für Mathematik Bonn\\
Germany}
\urladdr{https://sites.google.com/view/oanapadurariu/home}
\email{oana.padurariu11@gmail.com}

\author{\sc Frederick Saia}
\address{Frederick Saia \\
University of Illinois Chicago\\
USA}
\urladdr{https://fsaia.github.io/site/}
\email{fsaia@uic.edu}

\maketitle
\begin{abstract}
 We work towards completely classifying all bielliptic Shimura curves $X_0^D(N)$ with nontrivial level $N$ coprime to $D$ and their bielliptic involutions, extending a result of Rotger that provided such a classification for level one. Combined with prior work, this allows us to determine the list of all relatively prime pairs $(D,N)$ for which $X_0^D(N)$ has infinitely many degree $2$ points. As an application, we use these results to make progress on determining which curves $X_0^D(N)$ have sporadic points. Using tools similar to those that appear in this study, we also determine all of the geometrically trigonal Shimura curves $X_0^D(N)$ with $\gcd(D,N)=1$ (none of which are trigonal over $\Q$). 
\end{abstract}


\section{Introduction}

The study of low degree points on classical families of modular curves over $\mathbb{Q}$, including $X_0(N)$ and $X_1(N)$, is a subject of great modern interest in number theory, by virtue of its relationship to the study of rational isogenies and torsion points of elliptic curves over number fields. In this work, we are interested in the existence of low degree points on Shimura curves, which provide a generalization of modular curves. Shimura curves parameterize abelian surfaces with quaternionic multiplication and extra structure, and have canonical algebraic models as given in \cite{Sh67}. In particular, our study is focused on quadratic points on the family $X_0^D(N)$ of Shimura curves over $\mathbb{Q}$.

The $D=1$ case of $X_0^1(N) \cong Y_0(N)$ recovers the elliptic modular curve setting. Whereas the rational points on $Y_0(N)$ were the subject of careful study in \cite{Maz78}, we have that $X_0^D(N)(\mathbb{R}) = \emptyset$ when $D>1$ \cite[Thm. 0]{Sh75}. Therefore, one first asks about the degree two points on these curves, and in this work we are specifically interested in which curves $X_0^D(N)$ have \emph{infinitely many} quadratic points. We set the following notation:

\begin{definition}
    Let $X$ be a curve over a number field $F$. The \textbf{arithmetic degree of irrationality} of $X$ is the positive integer
    \[ \textnormal{a.irr}_F(X) := \text{min}\left\{ d : \left(\bigcup_{[L : F] = d} X(L) \right) \text{ is infinite}\right \}. \]
\end{definition}

We are interested in which pairs $(D,N)$, with $N$ coprime to $D$, have $\text{a.irr}_{\mathbb{Q}}(X_0^D(N)) \leq 2$. A curve can have infinitely many rational points only if it is of genus at most one by Faltings' Theorem \cite{Fal83}. Recall the following definition:

\begin{definition}
    For a curve $X$ over a  number field $F$, we say that $X$ is \textbf{bielliptic (over $F$)} if there exists an elliptic curve $E$ over $F$ and a degree $2$ map $X \rightarrow E$ over $F$. We say that $X$ is \textbf{geometrically bielliptic} if it is bielliptic over some finite extension of $F$. 
\end{definition}

A result of Hindry \cite[Remarque, p. 221]{Hindry} and of Harris--Silverman \cite[Corollary 3]{HaSi91}\footnote{Hindry originally proved the result conditional on the Mordell--Lang conjecture, which was proven by Faltings \cite{Fal91} a few years later. The result was reproven by Harris--Silverman following the announcement of and similarly using Faltings' theorem.} states that if $g(X) \geq 2$ and $\textnormal{a.irr}_F(X) = 2$, then $X$ is either hyperelliptic or is bielliptic with a degree $2$ map to an elliptic curve over $F$ of positive rank.\footnote{A geometric analogue for degree $3$ follows from work of Abramovich--Harris \cite[Theorem 1]{AH91}: if $$\min_{F'/F \text{ finite}}\textnormal{a.irr}_L(X_{F'}) = 3,$$ then $X_{\overline{F}}$ is trigonal or is trielliptic with a degree $3$ map to an elliptic curve of positive rank. (There are some known errors in \cite{AH91}, regarding which we refer to \cite[p.236]{DF93} and \cite[p.4]{KV22}, but they do not interfere with this result.) This pattern fails for degrees greater than $3$; see \cite{DF93}, and see \cite{KV22} for further results in this vein.} 

In the $D=1$ case, all hyperelliptic modular curves $X_0(N)$ of genus at least $2$ were determined by Ogg \cite{Ogg74}, and all bielliptic modular curves $X_0(N)$ were determined by Bars \cite{Bars99}. Further, Bars determined which bielliptic curves $X_0(N)$ have a bielliptic quotient of positive rank over $\mathbb{Q}$, and thus determined all curves $X_0(N)$ with $\text{a.irr}_{\mathbb{Q}}(X_0(N)) = 2$ \cite[Thm. 4.3]{Bars99}. 

We therefore restrict ourselves to the $D>1$ case from this point onwards. Voight \cite{Voight09} listed all $(D,N)$ for which $X_0^D(N)$ has genus zero:
\[ \{(6,1), (10,1), (22,1)  \}, \]
and genus one:
\[ \{(6,5), (6,7), (6,13), (10,3), (10,7), (14,1), (15,1), (21,1), (33,1), (34,1), (46,1)   \}. \]

Work of Ogg \cite{Ogg83} and of Guo--Yang \cite{GY17} determines all $(D,N)$ for which $X_0^D(N)$ is hyperelliptic over $\mathbb{Q}$:
\begin{align*} 
\{ &(6,11), (6,19), (6,29), (6,31), (6,37), (10,11),(10,23), (14,5), (15,2), \\
& (22,3), (22,5), (26,1), (35,1), (38,1), (39,1), (39,2) , (51,1), (55,1), (58,1), (62,1), \\
& (69,1), (74,1), (86,1), (87,1), (94,1), (95,1), (111,1), (119,1), (134,1), \\
& (146,1), (159,1), (194,1), (206,1) \}. 
\end{align*}

Rotger \cite[Theorem 7]{Rotger02} determined all of the bielliptic Shimura curves $X_0^D(1)$. Those which are of genus at least $2$ and are not also hyperelliptic are as follows:
\begin{align*}  
D \in \{&57,65,77,82,85,106,115,118 ,122,129,143,\\
&166,178,202,210,215,314, 330,390,462,510,546\}. 
\end{align*}

Moreover, Rotger determined which of these curves have bielliptic quotients of positive rank over $\mathbb{Q}$, and thus completed the determination of Shimura curves $X_0^D(1)$ with $\textnormal{a.irr}_{\mathbb{Q}}(X_0^D(1)) = 2$ \cite[Theorem 9]{Rotger02}. Excluding the curves that are either hyperelliptic or of genus at most one, we are left with 
\[ D \in \{57,65,77,82,106,118 ,122,129,143,166,210,215,314, 330,390,510,546 \}. \] 

In this work, we extend Rotger's results to the $N>1$ case. The following result settles the determination of the geometrically bielliptic curves $X_0^D(N)$ and their bielliptic quotients, aside from two possible exceptions:

\begin{theorem}\label{theorem: main_thm}
Let $D$ be an indefinite rational quaternion discriminant and $N > 1$ coprime to $D$.  If $X_0^D(N)$ is geometrically bielliptic then any bielliptic involution is one of the Atkin--Lehner involutions $w_m$, and $(D,N,m)$ is listed in \cref{table: squarefree} or \cref{table: non-squarefree}, possibly except when $(D,N)$ is one the following two pairs:
\[
(6,25),(10,9).
\]

\end{theorem}

For $(6,25)$ and $(10,9)$, although there are bielliptic involutions of Atkin--Lehner type, we remain unsure of whether the Shimura curve $X_0^D(N)$ has any bielliptic involutions that are not Atkin--Lehner involutions. See \cref{genus_5_not_squarefree_rmk} for comments on these two curves. 

To reach this result, we begin with background on Shimura curves and their local points in \S $2$ and \S $3$, and relevant algebro-geometric results in \S $4$. The proof then comes in \S $5$, with the following structure:
\begin{itemize}
    \item We reduce to finitely many candidate curves, using an explicit lower bound on the gonality of $X_0^D(N)$ based on an explicit lower bound on $g(X_0^D(N))$ and a result of Abramovich (see \cref{theorem: Abr}) that relates the gonality and the genus of a Shimura curve. 
    \item For most of these candidates, we are able in \S 5.1 to use results on $\text{Aut}(X_0^D(N))$ to restrict consideration to Atkin--Lehner involutions of $X_0^D(N)$. 
    \item We then determine which candidate curves have a genus one Atkin--Lehner quotient, and work to decide which such quotients are elliptic curves over $\mathbb{Q}$. 
\end{itemize}

In the course of our proof of \cref{theorem: main_thm}, we also determine which genus one Atkin--Lehner quotients are elliptic curves \emph{of positive rank} over $\mathbb{Q}$.

\begin{theorem}\label{delta_eq_2_biell_thm}
Suppose that $N>1$ is relatively prime to $D$, that $g(X_0^D(N)) \geq 2$ and that $X_0^D(N)$ is not hyperelliptic. Then $\textnormal{a.irr}_{\mathbb{Q}}(X_0^D(N)) = 2$, necessarily by virtue of $X_0^D(N)$ being bielliptic with a degree $2$ map to an elliptic curve of positive rank over $\mathbb{Q}$, if and only if 
    \begin{align*} 
    (D,N) \in \{&(6,17),(6,23),(6,41),(6,71),(10,13), \\
    &(10,17),(10,29),(22,7),(22,17)\}.
    \end{align*}
\end{theorem}

\cref{delta_eq_2_biell_thm} completes the determination of curves $X_0^D(N)$ with infinitely many degree $2$ points. (See also the rephrasing to this aim as \cref{delta_eq_2}.) 

As an application of our main results, in \S 6 we improve on a result of the second named author from \cite[\S 10]{Saia24} (recalled as \cref{sp_thm_Saia}) concerning sporadic points on the Shimura curves $X_0^D(N)$ and $X_1^D(N)$. An abridged version of our main result from this section is as follows; see \cref{sporadic_thm} for the full statement. 

\begin{theorem}
    \begin{enumerate}
        \item For all but at most $129$ relatively prime pairs $(D,N)$ with $D>1$ and $\textnormal{gcd}(D,N) = 1$, the Shimura curve $X_0^D(N)$ has a sporadic CM point. For at least $73$ of these pairs, this curve has no sporadic points.
        \item For all but at most $321$ relatively prime pairs $(D,N)$ with $D>1$ and $\textnormal{gcd}(D,N) = 1$, the Shimura curve $X_1^D(N)$ has a sporadic CM point. For at least $58$ of these pairs, this curve has no sporadic points.
    \end{enumerate}
\end{theorem}

Finally, in \cref{trigonal_section}, we prove that there are exactly $5$ geometrically trigonal Shimura curves $X_0^D(N)$ with $\gcd(D,N)=1$. A curve $X_0^D(N)$ with $D>1$ cannot have an odd degree map to $\mathbb{P}^1_{\mathbb{Q}}$, as these curves have no real points, so these curves \emph{are not} trigonal over $\Q$. See \cref{trigonal_def} for the relevant definition and \cref{proposition: trigonal-prop} for the main result.

All computations described in this paper were performed using the Magma computer algebra system \cite{magma}, and all relevant code can be found in \cite{Rep}. 

\subsection*{Acknowledgements}
It is a pleasure to thank Eran Assaf, Andrea Bianchi, Pete L. Clark, Samuel Le Fourn, Davide Lombardo, Ciaran Schembri, and John Voight for helpful conversations. We thank Noam Elkies for a conversation that led us to recognize an error in our result on trigonal curves $X_0^D(N)$ in a prior version of this paper. We thank Pete L. Clark and Pieter Moree for useful comments on an earlier version. We are also thankful for the feedback of three anonymous referees, which helped us improve the article. O.P. is very grateful to the Max-Planck-Institut für Mathematik Bonn for their hospitality and financial support.

\section{Some background on Shimura curves}

Let $D$ be an indefinite rational quaternion discriminant, i.e., the product of an even number of distinct prime numbers, and let $B_D$ denote the unique (up to isomorphism) quaternion algebra over $\mathbb{Q}$ of discriminant $D$. Let $N$ be a positive integer which is relatively prime to $D$. The curve $X_0^D(N)$ can then be described up to isomorphism as the coarse space for either of the following moduli problems:
\begin{itemize}
    \item tuples $(A,\iota,\lambda,Q)$, where $A$ is an abelian surface, $\iota: \mathcal{O} \hookrightarrow \text{End}(A)$ is an embedding of a maximal order $\mathcal{O} \subseteq B_D$ into the endomorphism ring of $A$, $\lambda$ is a principal polarization of $A$ which is compatible\footnote{See \cite[\S 2.1]{Saia24} for details.} with $\iota$ and $Q \leq A[N]$ is a cyclic $\mathcal{O}$-submodule of rank $2$ as a module over $\mathbb{Z}/N\mathbb{Z}$. 
    \item triples $(A,\iota,\lambda)$, where $A$ is an abelian surface, $\iota : \mathcal{O}_N \hookrightarrow \text{End}(A)$ is an embedding of an Eichler order $\mathcal{O}_N \subseteq B_D$ of level $N$ and $\lambda$ is a principal polarization compatible with $\iota$.
\end{itemize}
Similar to the first interpretation above, the curve $X_1^D(N)$ parametrizes triples $(A,\iota,\lambda,P)$ where $(A,\iota,\lambda)$ is as in the first interpretation above and $P \in A[N]$ is of order $N$. We call the data of $(A,\iota)$ as in any of the interpretations above a \textbf{QM abelian surface}; Shimura curves parametrize polarized QM abelian surfaces with additional structure. 

There is a natural covering map $X_1^D(N) \rightarrow X_0^D(N)$ of degree $\frac{\phi(N)}{2}$. On the level of moduli, this is described as $[(A,\iota,\lambda,P)] \mapsto [(A,\iota, \lambda,\iota(\mathcal{O}) \cdot P)]$, where $\iota(\mathcal{O}) \cdot P$ is the $\mathcal{O}$-cyclic subgroup of $A[N]$ generated by $P$. While the curves $X_0^D(N)$ will be the main interest in this work, the curves $X_1^D(N)$ will also come into play when we study sporadic points on both families in \S 6. 

\subsection{CM points and embedding numbers} 
In this work, we will mainly be concerned with the arithmetic of these Shimura curves as algebraic curves over $\mathbb{Q}$, and not specifically with their moduli interpretations. The place where the moduli interpretation will be most relevant will be in our discussion of CM points. 

\begin{definition}
    Let $K$ be an imaginary quadratic number field. A point $x \in X_0^D(N)(\overline{\Q})$ or $X_1^D(N)(\overline{\Q})$ is a \textbf{$K$-CM point} if it is induced by a QM abelian surface $(A,
    \iota)$ with either of the following equivalent properties:
    \begin{itemize}
        \item $A$ is geometrically isogeneous to $E^2$, where $E$ is an elliptic curve with $K$-CM.
        \item The ring $\text{End}(A,\iota)$ of $\iota(\mathcal{O})$-equivariant endomorphisms of $A$ is an order in $K$.
    \end{itemize}
    We call $x$ a \textbf{CM point} if it is a $K$-CM point for some imaginary quadratic field $K$.
\end{definition}

\begin{remark}
    It is common to define the notion of a CM point on a modular or Shimura curve first for points over $\C$ from the complex-analytic perspective. CM points are always algebraic, though, so the above definition is justified. More specifically, for any $K$-CM point $x \in X_0^D(N)$, the corresponding residue field $\mathbb{Q}(x)$ is either a ring class field of the CM field $K$, or an index $2$ subfield of a ring class field (necessarily totally complex if $D>1$) by \cite[Main Thm. 1]{Sh67}. 
\end{remark}
As for the modular curves $X_0(N)$, one often seeks to attach a specified imaginary quadratic order $R$ of $K$ to a $K$-CM point on the Shimura curve $X_0^D(N)$, and it seems that our definition above provides a clear way to do so: take $R := \text{End}(A,\iota)$. The catch here is that we have provided two moduli interpretations for points on $X_0^D(N)$: should we take $\iota$ to be a QM structure by a maximal order (such that our moduli datum also has the information of a certain $\mathcal{O}$-cyclic subgroup of $A[N]$), or by an Eichler order of level $N$? Both are reasonable choices, and they provide the same set of $K$-CM points, but it is important to distinguish between these choices as they alter the notion of an $R$-CM point for a fixed $R$ (see \cite[Remark 2.9]{Saia24} for a related remark). 

We will need to assign imaginary quadratic orders to CM points in two places in this paper: when applying results of Ogg \cite{Ogg83} (see \cref{thm_local_emb_Ogg} and \cref{thm_Ogg_fixed_pts}) to count the fixed points of Atkin--Lehner involutions, and when applying a result of Gonz\'alez--Rotger \cite[Cor 5.14]{GR06} on the residue fields of CM points on Atkin--Lehner quotients of $X_0^D(N)$. All of these results use the correspondence between $R$-CM points on $X_0^D(N)$ and optimal embeddings of orders in quadratic number fields into Eichler orders in $B_D$ -- see \cref{opt_emb_def} and \cref{Eichler-Thm}. This particular correspondence makes use of the second moduli interpretation we gave for $X_0^D(N)$, involving Eichler orders of level $N$. Thus, when we specify CM orders in this work, we will mean according to this interpretation. 

We recall here the genus formula for $X_0^D(N)$, which can be found, for example, in \cite[Thm. 39.4.20]{Voight21}. First, some notation: Let $\varphi$ and $\psi$ be the multiplicative functions such that, for every prime $p$ and positive integer $k$,
\[
\varphi(p^k) = p^k - p^{k-1} \quad \text{ and } \quad \psi(p^k) = p^k + p^{k-1},
\]
and let $\left(\frac{\cdot}{\cdot}\right)$ denote the Kronecker quadratic symbol.

\begin{proposition}
For $D$ the discriminant of an indefinite quaternion algebra over $\Q$, $N$ a positive integer coprime to $D$ and $k \in \{3,4\}$, define
\[
e_k(D,N) \colonequals \prod_{p|D} \left(1 - \left(\frac{-k}{p}\right)\right) \prod_{q \parallel N} \left(1 + \left(\frac{-k}{q}\right)\right) \prod_{q^2|N}\delta_q(k) ,
\]
where 
\[   
\delta_q(k) = 
     \begin{cases}
       2 &\quad\text{if } \left(\frac{-k}{q}\right) = 1, \\
       0 &\quad\text{otherwise.}  \\
     \end{cases}
\]
Then, for $D > 1$,
\[
g(X_0^D(N)) = 1 + \frac{\varphi(D)\psi(N)}{12} - \frac{e_4(D,N)}{4} - \frac{e_3(D,N)}{3}.
\]
\end{proposition}

We next recall a result of Eichler which relates local embedding numbers of quadratic orders into quaternion orders to global embedding numbers. These will be relevant both in results related to local points on Shimura curves and to fixed points of Atkin--Lehner involutions. We first set relevant definitions and notation. 

\begin{definition}\label{opt_emb_def}
    Let $\mathcal{O}$ be an Eichler order in $B_D$, let $K$ be a quadratic number field and let $R$ be an order in $K$. An \textbf{optimal embedding} of $R$ into $\mathcal{O}$ is an injection $\iota_K : K \hookrightarrow B_D$ such that $\iota_K^{-1}(\mathcal{O}) = R$.

    Two optimal embeddings $\iota_1$ and $\iota_2$ are \textbf{equivalent} if there is some $\gamma \in \mathcal{O}^\times$ such that $\iota_2(\alpha) = \gamma \iota_1(\alpha) \gamma^{-1}$ for all $\alpha \in K$. 
\end{definition}

For the remainder of the paper, let $\mathcal{O}_N$ denote a fixed Eichler order of level $N$ in $B_D$. For $p$ a rational prime, we let $(B_D)_p$ denote the localization of $B_D$ at $p$ and let $(\mathcal{O}_N)_p$ be the localization of $\mathcal{O}_N$ at $p$. We let $\nu(R,\mathcal{O}_N)$ denote the number of inequivalent optimal embeddings of $R$ into $\mathcal{O}_N$, and let $\nu_p(R,\mathcal{O}_N)$ denote the number of inequivalent optimal embeddings of $R_p$ into $(\mathcal{O}_N)_p$. We then have the following theorem of Eichler (see \cite[Theorem 1]{Ogg83}). 

\begin{theorem}[Eichler]\label{Eichler-Thm} 
Let $R$ be an order in a quadratic number field, and let $h(R)$ be the class number of $R$. Then
\[
\nu(R,\mathcal{O}_N) = h(R) \prod_{p\mid DN}\nu_p(R,\mathcal{O}_N).
\]
More precisely, suppose we are given for each $p \mid DN$ an equivalence class of optimal embeddings $R_p \hookrightarrow (\mathcal{O}_N)_p$. Then there are exactly $h(R)$ inequivalent optimal embeddings $R \hookrightarrow \mathcal{O}_N$ which are in the given local classes. 
\end{theorem}

The following result of Ogg allows us to compute these local embedding numbers, and hence to compute global embedding numbers. 

\begin{theorem}\cite[Theorem 2]{Ogg83}\label{thm_local_emb_Ogg}
Let $f$ be the conductor of $R$. Then $\nu_p := \nu_p(R,\mathcal{O}_N)$ is given below, according to various cases; divisibility is to be understood as in $\Z_p$, and $\psi_p$ is the multiplicative function with $\psi_p(p^k) = p^k(1+1/p)$ and $\psi_p(n) = 1$ if $p \nmid n$.   
\begin{enumerate}[(i)]
    \item If $p \mid D$, then $\nu_p = 1 - \left(\frac{R}{p}\right)$.
    \item If $p \parallel N$, then $\nu_p = 1 + \left(\frac{R}{p}\right)$.
    \item Suppose $p^2 \mid N$. 
    \begin{enumerate}[(a)]
        \item If $(pf)^2 \mid N$, then
        \[   
        \nu_p = 
     \begin{cases}
       2\psi_p(f) &\quad\text{if } \left(\frac{L}{p}\right) = 1, \\
       0 &\quad\text{otherwise.}  \\
     \end{cases}
\]
    \item If $pf^2 \parallel N$ (say $k$ is such that $p^k \parallel f$), then
     \[   
        \nu_p = 
     \begin{cases}
       2\psi_p(f) &\quad\text{if } \left(\frac{L}{p}\right) = 1, \\
       p^k &\quad\text{if } \left(\frac{L}{p}\right) = 0, \\
       0 &\quad\text{if } \left(\frac{L}{p}\right) = -1. \\
     \end{cases}
\]

\item If $f^2 \parallel N$ (say $k$ is such that $p^k \parallel f$), then

\[
\nu_p = p^{k-1}\left(p+1+\left(\frac{L}{p}\right)\right).
\]
\item If $pN \mid f^2$, then

 \[   
        \nu_p = 
     \begin{cases}
       p^k+p^{k-1} &\quad\text{if } N \parallel p^{2k}, \\
       2p^k &\quad\text{if } N \parallel p^{2k+1}.  \\
     \end{cases}
\]

    \end{enumerate}
\end{enumerate}
\end{theorem}

\subsection{The Atkin--Lehner group}
Consider the group of norm $1$ units in our Eichler order:
\[ \mathcal{O}_N^1 := \{ \gamma \in \mathcal{O}_N \mid \text{nrd}(\gamma) = 1\} . \]
Elements of the normalizer subgroup
\[ N_{{B_D}^\times_{> 0}}(\mathcal{O}_N^1) := \{\alpha \in {B_D}^\times \mid \text{nrd}(\alpha) > 0 \text{ and } \alpha^{-1} \mathcal{O}_N^1 \alpha = \mathcal{O}_N^1 \} \]
naturally act on $X_0^D(N)$ via action on the QM-structure $\iota$, with $\mathbb{Q}^\times \hookrightarrow {B_D}^\times$ acting trivially. The full \textbf{Atkin--Lehner group} is the group 
\[ W_0(D,N) := N_{{B_D}^\times_{> 0}}(\mathcal{O}_N^1) / \mathbb{Q}^\times \mathcal{O}_N^1 \subseteq \text{Aut}(X_0^D(N)). \]
The group $W_0(D,N)$ is a finite abelian $2$-group, with an \textbf{Atkin--Lehner involution} $w_m$ associated to each Hall divisor of $DN$. That is,
\[ W_0(D,N) = \{w_m \mid m \parallel DN\} \cong \left(\mathbb{Z}/2\mathbb{Z}\right)^{\omega(DN)}, \]
where $\omega(DN)$ denotes the number of distinct prime divisors of $DN$. 

The following result says that fixed points of Atkin--Lehner involutions correspond to optimal embeddings of specific imaginary quadratic orders, i.e., to CM points by specific orders.

\begin{theorem}\cite[p. 283]{Ogg83}\label{thm_Ogg_fixed_pts}
    The fixed points by the Atkin--Lehner involution of level $m \parallel DN$ acting on $X_0^D(N)$ are \textnormal{CM} points by
\[   
R = 
     \begin{cases}
       \Z[i],\Z[\sqrt{-2}] &\quad\text{if } m = 2; \\
       \Z\left[\frac{1+\sqrt{-m}}{2}\right], \Z[\sqrt{-m}] &\quad\text{if } m \equiv 3 \;(\bmod \; 4); \\
       \Z[\sqrt{-m}] &\quad\text{ otherwise}.\\
     \end{cases}
\]
The number of fixed points corresponding to any of the previous orders $R$ is given by
\[
h(R) \prod_{p|\frac{DN}{m}} \nu_p(R,\mathcal{O}_N).
\]
\end{theorem}

Consequently, the Fricke involution $w_{DN}$ always has fixed points:
\begin{corollary}
Assume $D > 1$. Then the number of fixed points of $X_0^D(N)$ by the Fricke involution is 
\[   
\#X_0^D(N)^{w_{DN}} = 
     \begin{cases}
    h\left(\Z\left[\frac{1+\sqrt{-DN}}{2}\right]\right) + h(\Z[\sqrt{-DN}]) &\quad\text{if } DN \equiv 3 \;(\bmod \; 4), \\
       h(\Z[\sqrt{-DN}]) &\quad\text{ otherwise}.\\
     \end{cases}
\] 
\end{corollary}

\section{Local points} 

To prove that $X_0^D(N)/\langle w_m \rangle$ has no $\mathbb{Q}$-rational points, it is sufficient to prove the non-existence of points over $\mathbb{R}$ or over $\mathbb{Q}_p$ for some prime $p$. We will use such arguments later in determining which genus one Atkin--Lehner quotients are in fact elliptic curves over $\mathbb{Q}$, so in this section we recall results on local points on these quotients. 

\subsection{Real points}

We mentioned in the introduction that when $D>1$, we have
\[
X_0^D(N)(\mathbb{R}) = \emptyset.
\]
In \cite{Ogg83}, Ogg completes a study of real points on quotients $X_0^D(N)/\langle w_m \rangle$ for $m \parallel DN$. As $\left(X_0^D(N)/\langle w_m \rangle\right)(\mathbb{R})$ is a real manifold of dimension one, it is a disjoint union of circles. The number of connected components is related to the number of classes of certain optimal embeddings of orders in the real quadratic field $\mathbb{Q}(\sqrt{m})$ into $\mathcal{O}_N$; we summarize Ogg's results here:
\begin{theorem}\cite[Proposition 1, Theorem 3]{Ogg83}\label{thm_ogg_real_pts}
   Let $D > 1$ and $m \parallel DN$, and set 
   \[
    \nu(m) = \sum_{R} h(R) \prod_{p|\frac{DN}{m}} \nu_p(R,\mathcal{O}_N),
    \]
    where $R$ ranges over $\Z[\sqrt{m}]$, and also $\Z\left[\frac{1+\sqrt{m}}{2}\right]$ if $m \equiv 1 \pmod{4}$. 
    
    Let $\#(m)$ denote the number of connected components of $\left(X_0^D(N)/\langle w_m \rangle \right)(\mathbb{R})$. If $m$ is a square, then $\#(m) = 0$, i.e., this quotient has no real points. If $m$ is not a square, then 
    \[ \#(m) = \nu(m)/2, \]
    unless $\nu(m) > 0, i \in \mathcal{O}_N, DN = 2t$, with $t$ odd, $m=t$ or $2t$, and $x^2-my^2= \pm 2$ is solvable with $x,y \in \Z$, in which case 
    \[ \#(m) = \left(\nu(m)+2^{\omega(DN)-2}\right)/ 2. \]
\end{theorem}

\subsection{$p$-adic points}
We next recall a result on $\mathbb{Q}_p$-rational points on $X_0^D(N)$ and certain Atkin--Lehner quotients thereof, coming from work of Ogg \cite{Ogg85}.

\begin{theorem}\cite[Th{\'e}or{\`e}me, p. 206]{Ogg85}
Let $p$ be a prime dividing $D$, and let $\mathcal{O}_N$ be an Eichler order of level $N$ in $B_D$. Suppose that $m \parallel \frac{DN}{p}$ and $m > 1$.

Now let $\widehat{B}$ denote the definite (ramified at infinity) quaternion algebra over $\mathbb{Q}$ of discriminant $D/p$, let $\widehat{\theta}$ be a choice of Eichler order of level $N$ in $\widehat{B}$ and let $h = h\left(\widehat{\theta}\right)$ be the class number of $\widehat{\theta}$. Let $\widehat{\theta}_1, \ldots, \widehat{\theta}_h$ denote the inequivalent Eichler orders of level $N$ in $\widehat{B}$. 
\begin{enumerate}[(i)]
    \item $X_0^D(N)(\Q_p)$ is non-empty if and only if $p=2$ and $\sqrt{-1} \in \mathcal{O}_N$, or $p \equiv 1 \pmod{4}$, $N = 1$ and $D = 2p$.
    \item If $X_0^D(N)(\Q_p)$ is empty, then $(X_0^D(N)/\langle w_m \rangle)(\Q_p)$ is non-empty if and only if one of the following holds:
    \begin{enumerate}[a)]
        \item $p=2$, $m = \frac{DN}{p}$ and $\sqrt{-2} \in \mathcal{O}_N$,
        \item $p > 2$, $\sqrt{-p} \in \mathcal{O}_N$, $\left(\frac{-m}{p} \right) = 1$, $\frac{DN}{p} \in \{ m, 2m\}$ and $\big( \, 8|(p+1)(m+1)$ if $\frac{DN}{p} = 2m$ and $2|N \big)$, or
        \item $p \equiv 1 \pmod{4}$, $\sqrt{-1} \in \widehat{\theta}_i$ for some $1 \leq i \leq h$, $\frac{DN}{p} \in \{ m, 2m\}$ and $\sqrt{-pm} \in \mathcal{O}_N$. 
    \end{enumerate}
    \item $(X_0^D(N)/\langle w_p\rangle) (\Q_p)$ is non-empty if and only if there is an index $1 \leq i \leq h$ such that $\widehat{\theta}_i$ contains $\sqrt{-p}$ or a root of unity not equal to $\pm 1$.
    \item If $X_0^D(N)(\Q_p)$ is empty, then $(X_0^D(N)/\langle w_{pm} \rangle) (\Q_p)$ is non-empty if and only if there is some $1 \leq i \leq h$ with $\sqrt{-m} \in \widehat{\theta}_i$, or $\sqrt{-1} \in \widehat{\theta_i}$ in the case of $m = 2$. 
\end{enumerate} 
\end{theorem}

\section{Gonalities and involutions of algebraic curves}


For a field $F$ and a curve $C/F$, the $F$-\textbf{gonality} $\gon_F(C)$ is defined to be the least degree of a non-constant morphism $f : C \rightarrow \P^1$ defined over $F$. For $\overline{F}$ an algebraic closure of $F$, the $\overline{F}$-gonality is also called the \textbf{geometric gonality}. 
We will use the following result to obtain an upper bound on the genus of $X_0^D(N)$:

\begin{theorem}{\cite[Theorem 1.1]{Abramovich96}} \label{theorem: Abr}
For a Shimura curve $X_0^D(N)$, we have
\[
 g(X_0^D(N)) \le \frac{200}{21} \gon_{\C} (X_0^D(N)) + 1.
\]   
\end{theorem}

One can compose a map to the projective line with any covering map, giving the following simple but useful result. 

\begin{proposition}
Let $f : X \rightarrow Y$ be a non-constant morphism of curves over $F$. Then
\[
\gon_F(X) \le \deg(f) \gon_F(Y).
\]
\end{proposition}

For instance, a degree $d$ cover of a (hyper)elliptic curve must have gonality at most $2d$.

\begin{corollary}\label{gon_cor}
If $X$ is a bielliptic curve over a field $F$, then $\gon_F(X) \le 4$. 
\end{corollary}

\begin{proposition}\cite[Proposition 1]{HaSi91}\label{HS_prop}
If $f: X \rightarrow Y$ is a non-constant morphism of curves over $F$ and $X$ is geometrically bielliptic, then $Y$ is geometrically hyperelliptic, geometrically bielliptic, or of genus at most one.    
\end{proposition}

\begin{lemma}\cite[Lemma 4.3]{BKX13}\label{not}
Consider a Galois cover $\phi : X \to Y$ of degree $d$ between two non-singular projective curves of genus $g_X \ge 2$ and $g_Y$, respectively. Suppose that $g_Y \ge 2$ or $d$ is odd.

\begin{enumerate}[(1)]
    \item Suppose that $2g_X+2 > d(2g_Y+2)$. Then $X$ is not geometrically hyperelliptic.
    \item Denote by $\#Y^{\sigma}$ the number of fixed geometric points of an involution $\sigma$ of $Y$. Suppose $2g_X - 2 > d \cdot \#Y^{\sigma}$ for any involution $\sigma$ on $Y$. Then, if $g_X \ge 6$, $X$ is not geometrically bielliptic.
    \item Suppose $2g_X-2 > d(2g_Y+2)$. Then, if $g_X \ge 6$, $X$ is not geometrically bielliptic. \label{not-bielliptic}
\end{enumerate}
\end{lemma}

\begin{lemma}{\cite[Lemma 5.(2)]{Rotger02}}\label{s_lemma}
Let $C/F$, $\text{char}(F) \ne 2$, be a bielliptic curve of genus $g$ with $\Aut(C) \cong C_2^s$ for some $s \ge 1$.

\begin{itemize}
    \item If $g$ is even, then $s \le 3$.
    \item If $g$ is odd, then $s \le 4$.
\end{itemize}

\end{lemma}

The following result follows from the Riemann--Hurwitz formula:

\begin{proposition} \label{proposition: genus_of_quotient}
    Let $\sigma$ be any involution on a smooth projective curve $X$ over an algebraically closed field $F$ of characteristic $0$, and let $\#X^{\sigma}$ denote the number of fixed points of $\sigma$. Then, we have the following genus formula:
    $$g(X/\langle\sigma \rangle) = \frac{1}{4}(2g(X)+2-\#X^{\sigma}).$$
\end{proposition}

\begin{remark}
A curve of genus $g$ is geometrically bielliptic if and only if there is an involution with $2g-2$ geometric fixed points.
\end{remark}

\begin{lemma}\cite[Lemma 4.3]{BKS23}\label{lemma: more-than-8}
Let $\sigma$ be an involution of $X$ with more than $8$ fixed points. Then either $\sigma$ is a bielliptic involution, or $X$ is not geometrically bielliptic.
\end{lemma}

\begin{lemma}\cite[Proposition 4.8]{BKS23} 
  Let $X$ be a curve of genus $g$ at least $6$ over a field of characteristic $0$. Assume that $\Aut(X)$ has a subgroup $H$ of order $2^t$ such that $2^t \nmid 2(g-1)$. Then either the bielliptic involution of $X$ is contained in $H$, or $X$ is not geometrically bielliptic.
\end{lemma}

The group $W_0(D,N)$ is a subgroup of $\Aut(X_0^D(N))$ of order $2^{\omega(DN)}$. We therefore have the following corollary:

\begin{corollary}\label{not-div}
Suppose that $g(X_0^D(N)) \geq 6$ and that $X_0^D(N)$ is geometrically bielliptic. If $g(X_0^D(N)) \not \equiv 1 \;(\bmod \; 2^{\omega(DN)-1})$, then the bielliptic involution is an Atkin--Lehner involution. 
\end{corollary}

\section{Proof of \cref{theorem: main_thm} and \cref{delta_eq_2_biell_thm}}
From \cref{theorem: Abr} and \cref{gon_cor}, we find that a geometrically bielliptic Shimura curve $X_0^D(N)$ must have genus $g(X_0^D(N)) \leq 39$. 

\begin{lemma}\cite[Lemma 10.6]{Saia24}\label{genus_bnd_lemma}
For $D>1$ an indefinite rational quaternion discriminant and $N \in \mathbb{Z}^+$ relatively prime to $D$, we have
\[ g(X_0^D(N)) > 1 + \frac{DN}{12}\left( \frac{1}{e^\gamma \log\log(DN) + \frac{3}{\log\log{6}}} \right)- \frac{7\sqrt{DN}}{3}.\]
\end{lemma}
With \cref{genus_bnd_lemma}, we find that if $DN > 78530$ then we have that $g(X_0^D(N)) > 39$ and thus $X_0^D(N)$ is not bielliptic.  

If $X_0^D(N)$ is geometrically bielliptic, then $X_0^D(1)$ must be geometrically bielliptic, geometrically hyperelliptic, or of genus $g(X_0^D(1)) \leq 1$ by \cref{HS_prop}. By prior results of Voight \cite{Voight09} (for genus at most one), Ogg \cite{Ogg83} (for geometrically hyperelliptic of genus at least $2$), and Rotger \cite{Rotger02} (for geometrically bielliptic), it follows that 
\begin{align*} 
D \in \{&6,10,14,15,21,22,26,33,34,35,38,39,46,51,55,57,58,62,65,69,74,\\
&77,82,85,86,87,94,95,106,111,115,118,119,122,129,134,143,146, \\
&159,166,178,194,202,206,210,215,314,330,390,462,510,546\}. 
\end{align*}
Note that all of these values of $D$ satisfy $\omega(D) = 2$. Using this and our genus bound \cref{genus_bnd_lemma}, we arrive at $357$ \textbf{candidate pairs} $(D,N)$. These are computed in the file \texttt{narrow{\_}to{\_}candidates.m} in \cite{Rep}, and comprise the list in \texttt{candidate{\_}pairs.m}. All but $56$ of these candidate pairs have squarefree level $N$. 


\subsection{Automorphisms of candidate pairs}

In this section, we prove for certain candidate pairs $(D,N)$ that $\Aut(X_0^D(N)) = W_0(D,N)$ is the group of Atkin--Lehner involutions, and for others we determine restrictions on the involutions in $\Aut(X_0^D(N))$. This will help us determine when this Shimura curve is \emph{not} bielliptic (over $\Q$), by restricting our study to bielliptic involutions in $W_0(D,N)$.

Our main tools come from \cite{KR08}, which extends work of \cite{Rotger02}. In particular, we make use of the following result:

\begin{lemma}\label{all_atkin_lehner_lemma}
Suppose that $D$ is an indefinite rational quaternion discriminant and $N$ is a squarefree integer with $\gcd(D,N) = 1$. Also suppose that $g \colonequals g(X_0^D(N)) \geq 2$. If any of the following statements holds:
\begin{enumerate}
    \item $e_3(D,N) = e_4(D,N) = 0$,
    \item $2 \mid DN$, for all primes $p \mid N$ we have $\left(\frac{-4}{p}\right) \ne -1$ and for at most one prime $p \mid D$ we have $\left(\frac{-4}{p}\right) = 1$.
    \item $3 \mid DN$, for all primes $p \mid N$ we have $\left(\frac{-3}{p}\right) \ne -1$ and for at most one prime $p \mid D$ we have $\left(\frac{-3}{p}\right) = 1$.
    \item $\omega(DN) = \textnormal{ord}_2(g-1)+2$,
\end{enumerate}
then $\Aut(X_0^D(N)) = W_0(D,N)$. 
\end{lemma}
\begin{proof}
    The first part is \cite[Thm 1.6 (i)]{KR08}, while the second and third parts are \cite[Thm 1.7 (i)]{KR08}. The fourth part follows from \cite[Thm 1.6 (iii)]{KR08}. 
\end{proof}

\begin{lemma}\label{all_atkin_lehner_lemma_2}
 Suppose that $(D,N)$ is a candidate pair with $N$ squarefree and $g(X_0^D(N)) \geq 2$. If either 
 \begin{itemize}
    \item $g$ is even and $\omega(DN) = 3$, or
    \item $g$ is odd and $\omega(DN) = 4$
 \end{itemize}
and $X_0^D(N)$ is geometrically bielliptic, then $\Aut(X_0^D(N)) = W_0(D,N)$. 
\end{lemma}
\begin{proof}
Our hypotheses on $N$ and the genus imply $\Aut(X_0^D(N)) \cong \left(\mathbb{Z}/2\mathbb{Z}\right)^s$ for some $s \geq r$ by \cite[Prop 1.5]{KR08}. The result then follows from \cref{s_lemma}.
\end{proof}


Using \cref{all_atkin_lehner_lemma} and \cref{all_atkin_lehner_lemma_2}, we explicitly compute in \texttt{narrow{\_}to{\_}candidates.m} that if $X_0^D(N)$ is geometrically bielliptic then $\Aut(X_0^D(N)) = W_0(D,N)$ for all of our candidate pairs $(D,N)$ with $N$ squarefree and $g(X_0^D(N)) \geq 2$ except for possibly the following $25$ pairs:
\begin{align*}
 \{ &( 10, 19 ),
    ( 10, 31 ),
    ( 10, 43 ),
    ( 10, 67 ),
    ( 10, 79 ),
    ( 10, 103 ),
    ( 21, 5 ),
    ( 21, 17 ), \\
    &( 21, 29 ),
    ( 22, 7 ),
    ( 22, 31 ),
    ( 33, 5 ),
    ( 33, 17 ),
    ( 34, 7 ),
    ( 34, 19 ), 
    ( 46, 7 ),\\
    & ( 55, 7 ), 
    ( 57, 5 ),
    ( 58, 7 ),
    ( 69, 5 ),
    ( 77, 5 ),
    ( 82, 7 ),
    ( 94, 7 ),
    ( 106, 7 ),
    ( 118, 7 )\}. 
\end{align*}
The genera of the curves $X_0^D(N)$ for the above $25$ pairs $(D,N)$ are among the set 
\[ \{5, 9, 13, 17, 21, 25, 29, 33, 37\}. \]

We can exclude $(D,N) \in \{(10,31),(33,5)\}$ from further consideration, as the two corresponding curves $X_0^D(N)$ have a bielliptic Atkin--Lehner involution (see \cref{table: squarefree}), which must be unique given that $g(X_0^D(N)) \ge 6$ for both. 

\begin{remark}\label{Humbert-remark}
By \cite{KMV11}, a curve of genus $5$ can have $0, 1, 2, 3$ or $5$ bielliptic involutions. The genus $5$ curves with $5$ biellliptic involutions are called \textbf{Humbert curves} and form a $2$-dimensional family.

For $(D,N) \in \{(21,5),(22,7)\}$, the curve $X_0^D(N)$ has genus $5$ and there are $3$ geometrically bielliptic involutions of Atkin--Lehner type (see \cref{table: squarefree}). Thus, either $X_0^D(N)$ has exactly $5$ bielliptic involutions (and so is a Humbert curve), or has exactly $3$ bielliptic involutions.

The Jacobian of a Humbert curve is geometrically isogenous to a product of five elliptic curves (\cite[Prop. 2.4]{FMZ18}), whereas we compute with Magma via Ribet's isogeny (\cref{Ribet_isog}) that the Jacobians of $X_0^{21}(5)$ and $X_0^{22}(7)$ each have a geometrically simple abelian surface as a factor in their isogeny decomposition. Thus, neither curve is a Humbert curve, and each has all involutions being Atkin--Lehner. 
\end{remark}

\begin{remark}\label{genus_5_not_squarefree_rmk}
For $(D,N) \in \{(6,25),(10,9)\}$, the curve $X_0^D(N)$ has genus $5$ and there exists one geometrically bielliptic involution of Atkin--Lehner type (see \cref{table: non-squarefree}). From \cref{Humbert-remark}, we then have that $X_0^D(N)$ has $1, 2, 3,$ or $5$ bielliptic involutions. Note that for both of these pairs we also have genus $2$ quotients by Atkin--Lehner involutions. For example $X_0^{6}(25)/\langle w_{2} \rangle$ and $X_0^{10}(9)/\langle w_2 \rangle$ both have genus $2$. Thus, by \cite[Lemma 4.11 (b)]{BKS23}, neither curve $X_0^D(N)$ can have exactly $2$ geometrically bielliptic involutions; this would imply that two Atkin--Lehner involutions do not commute. Therefore, if $X_0^D(N)$ does have a geometrically bielliptic involution which is not of Atkin--Lehner type for these pairs, then it has exactly $3$ or $5$ geometrically bielliptic involutions in total. If this is the case, then by \cite[Lemma 2.3]{KMV11} all of the bielliptic involutions commute, while by \cite[Lemma 4.11 (a)]{BKS23} the non-Atkin--Lehner bielliptic involutions do not commute with any Atkin--Lehner involutions other than the bielliptic one. 
\end{remark}

\subsection{Squarefree level $N$}\label{sqfree_section}
We have $301$ candidate pairs $(D,N)$ where $N$ is squarefree, listed in \texttt{sqfree{\_}candidate{\_}pairs.m}. 
For all except for $55$ of these pairs, 
there exists $m | DN$ such that 
\[
\#X_0^D(N)^{w_m} \ne 2g(X_0^D(N))-2 \quad \text{ and } \quad  \#X_0^D(N)^{w_m} > 8.
\]
Thus by \cref{lemma: more-than-8}, we can conclude that these $246$ pairs do not correspond to bielliptic Shimura curves.
We know that $41$ of the remaining $55$ admit at least one bielliptic Atkin--Lehner involution by explicit genus computations, which we perform in the file \texttt{genus{\_}1{\_}quotients{\_}and{\_}ranks.m} using the \texttt{quot{\_}genus} defined in the file \texttt{quot{\_}genus.m}. 

We list below the $41$ pairs $(D,N)$ with $N$ squarefree for which $X_0^D(N)$ admits at least one bielliptic Atkin--Lehner involution:
\begin{align*} 
& D= 6, \quad \; \; N \in \{5,7,11,13,17,19,23,41,43,47,71 \},\\
& D = 10, \quad N \in \{3,7,13,17,29,31 \}, \\
& D = 14, \quad N \in \{ 3,5,13,19  \}, \\ 
& D = 15, \quad N \in \{ 2,7,11,13,17  \}, \\ 
& D = 21, \quad N \in \{ 2,5,11  \}, \\ 
& D = 22, \quad N \in \{3,7,17  \}, \\ 
& D = 26, \quad N = 5, \\ 
& D = 33, \quad N \in \{2,5,7 \}, \\ 
& D = 34, \quad N = 3, \\ 
& D = 35, \quad N \in \{2,3  \}, \\ 
& D = 38, \quad N = 3, \\ 
& D = 46, \quad N = 5.  
\end{align*}

More specifically, we list all triples $(D,N,m)$ such that $N$ is squarefree and such that $X_0^D(N)/\langle w_m \rangle$ has genus one in \cref{table: squarefree}. We attempt using \cref{not}(\ref{not-bielliptic}) to prove that the remaining $14$ are not bielliptic, which works in all cases except $(D,N) = (34,7)$. We handle this pair seperately:

\begin{lemma}\label{34-7-lemma}
The curve $X_0^{34}(7)$ is not geometrically bielliptic.
\end{lemma}
\begin{proof}
Suppose that $X_0^{34}(7)$ has a geometrically bielliptic involution $\sigma$. The quotient $X_0^{34}(7)/\langle w_{14}, w_{17} \rangle$ has genus $0$. By the Castelnuovo--Severi inequality, because the genus of $X_0^{34}(7)$ is $9$, and in particular is larger than $5$, it must be the case that the covering map to $X_0^{34}(7)/\langle w_{14}, w_{17} \rangle$ factors through $\sigma$. This is a Galois cover, and so it follows that $\sigma \in W_0(34,7)$ (more specifically, $\sigma \in \{ w_{14},w_{17}, w_{238}\})$. We find that this is not the case, as $X_0^{34}(7)$ has no genus one Atkin--Lehner quotients. 
\end{proof}

For a triple $(D,N,m)$ for which the quotient $X_0^D(N)/\langle w_m \rangle$ has genus one, we may list out all of the quadratic CM points on $X_0^D(N)$ using the results of \cite{GR06} or of \cite{Saia24}. For $N$ squarefree, we can then use \cite[Cor 5.14]{GR06} to determine, for a quadratic point $P$ on $X_0^D(N)$, the residue field of the image of $P$ on $X_0^D(N)/\langle w_m \rangle$ under the natural quotient map. Doing so, we determine rational CM points of genus one quotients $X_0^D(N)/\langle w_m \rangle$ with $N$ squarefree:

\begin{proposition}
    For each discriminant $\Delta_K$ of an imaginary quadratic field $K$ listed in the row in this table corresponding to the triple $(D,N,m)$, the quotient $X_0^D(N)/\langle w_m \rangle$ has a $\mathbb{Q}$-rational $K$-CM point.

    For each triple $(D,N,m)$ appearing in this table, we therefore have that $X_0^D(N)/\langle w_m \rangle$ is an elliptic curve over $\mathbb{Q}$, and hence $X_0^D(N)$ is bielliptic over $\mathbb{Q}$. 
\end{proposition}
\begin{proof}
    We determine that each listed quotient has $K$-CM points with residue field $\mathbb{Q}$ for the specified fields $K$ by explicit computation as described above. Quadratic CM points on the relevant curves $X_0^D(N)$ are listed using the methods of \cite{Saia24} with code from \cite{SaiaRep}, and computations determining residue fields on quotients are performed in the file \texttt{rationality{\_}by{\_}CM.m} in \cite{Rep} using Rotger's result. Because these quotients have a rational point and each have genus $1$ (they each appear in \cref{table: squarefree}), they are bielliptic quotients of $X_0^D(N)$ over $\Q$. 
\end{proof}

\begin{remark}
If $N$ is not squarefree, then the work of \cite{Saia24} can still provide a full list of quadratic CM points on $X_0^D(N)$. The squarefree restriction on this method comes from the determination in \cite[Cor. 5.14]{GR06} of the field of moduli of the image of a CM point on $X_0^D(N)$ under an Atkin--Lehner quotient. With the squarefree restriction on $N$, the work of \cite{GR06} is enough to list all of the quadratic CM points on $X_0^D(N)$. 
\end{remark}

\rowcolors{2}{white}{gray!20}
\begin{longtable}{|l|l|}
\caption{Rational CM points on genus one quotients $X_0^D(N)/\langle w_m \rangle$ with $N>1$ squarefree}\label{CM_points_table} \\
\hline
\rowcolor{gray!50}
$(D,N,m)$ & $\Delta_K$ such that $X_0^D(N)/\langle w_m \rangle$ has a $\mathbb{Q}$-rational $K$-CM point \\ \hline
$(6,5,15)$ & $-4$ \\ \hline
$(6,7,14)$ & $-3$ \\ \hline
$(6,13,26)$ & $-3$ \\ \hline
$(6,13,39)$ & $-4$\\ \hline
$(6,17,51)$ & $-4$ \\ \hline
$(6,17,102)$ & $-4, -19, -43, -67$ \\ \hline
$(6,23,138)$ & $-19, -43, -67$\\ \hline
$(6,41,246)$ & $-4, -43, -163$\\ \hline
$(6,71,426)$ & $-67, -163$\\ \hline
$(10,3,10)$ & $-3$\\ \hline
$(10,3,15)$ & $-8$\\ \hline
$(10,13,130)$ & $-3, -43$ \\ \hline
$(10,17,170)$ & $-8, -43, -67$\\ \hline
$(10,29,290)$ & $-67$ \\ \hline
$(14,3,21)$ & $-8$\\ \hline
$(14,3,42)$ & $-8, -11$ \\ \hline
$(14,5,35)$ & $-4$\\ \hline
$(14,5,70)$ & $-4 -11$\\ \hline
$(14,13,182)$ & $-4, -43$ \\ \hline
$(14,19,266)$ & $-8, -67$\\ \hline
$(15,2,30)$ & $-7$ \\ \hline
$(15,7,105)$ & $-3, -7$\\ \hline
$(15,13,195)$ & $-3, -43$\\ \hline
$(15,17,255)$ & $-43, -67$\\ \hline
$(21,2,21)$ & $-4, -7$ \\ \hline
$(21,2,42)$ & $-4, -7$ \\ \hline
$(21,5,105)$ & $-4$\\ \hline
$(21,11,231)$ & $-7, -43$ \\ \hline
$(22,7,154)$ & $-3$\\ \hline
$(22,17,374)$ & $-4, -67$\\ \hline
$(33,2,66)$ & $-4$ \\ \hline
$(33,7,231)$ & $-3$ \\  \hline
$(35,2,35)$ & $-8$ \\ \hline
$(46,5,230)$ & $-4$ \\ \hline
\end{longtable}

In the thesis work of Nualart-Riera \cite{NR15}, one finds defining equations for the Atkin--Lehner quotients of $X_0^D(N)$ for $(D,N)$ in
\[
\{(6,5),(6,7),(6,11),(10,3),(10,7),(10,9),(22,5),(22,7) \}.
\]

As defining equations for Shimura curves and their quotients are difficult to compute, one may use Ribet's isogeny to compute the rank of $\Jac(X_0^D(N)/\langle w_m \rangle)$ over $\mathbb{Q}$:

\begin{theorem}\cite{Ribet90},\cite{BD96}\label{Ribet_isog}
Let $X_0^D(N)$ be a Shimura curve with squarefree level $N$. Then there exists an isogeny defined over $\Q$

\begin{equation}
    \psi : J_0(DN)^{D-\new} \rightarrow \Jac(X_0^D(N)),
\end{equation}
such that, for each $w_m(D,N) \in W_0(D,N)$, we have

\begin{equation}
    \psi^*(w_m(D,N)) = (-1)^{\omega(\gcd(D,m))}w_m(1,DN) \in \Aut_{\Q}(J_0(DN)).
\end{equation} 
\end{theorem}

\begin{remark}\label{ribet_general_level_rmk}
While the squarefree restriction on the level $N$ appears in Ribet's paper \cite{Ribet90} and other recommended sources such as \cite{Helm07}, it is known that the result can be generalized: for arbitrary $N$ coprime to $D$, the Jacobian $\Jac(X_0^D(N))$ is isogenous to $J_0(DN)^{D-\new}$. The relevant input to get the generalization is the existence of Hecke-invariant isomorphisms between spaces of modular forms attached to the Shimura curve, and a direct sum of such spaces attached to $X_0(dN)$ for $d \mid D$. Work of Hijikata--Pizer--Shemanske \cite{HPS89a, HPS89b}, and also recent work of Martin \cite[Theorem 1.1 (ii)]{Mar20} using an alternate approach, provide exactly this for general level in our situation. We will therefore be justified to use the existence of Ribet's isogeny in the next section for non-squarefree level $N$. 
\end{remark}

As we usually do not have defining equations for Shimura curves, we are going to use Ribet's isogeny in the following way. Magma is able to provide a decomposition (up to isogeny)
\[
J_0(DN) \sim (J_0(DN))^{m+} \oplus (J_0(DN))^{m-}.
\]
We pick one of the two subspaces according to the parity of $\omega(\gcd(D,m))$, and after taking the $D$-new part we obtain an abelian variety isogenous to $\Jac(X_0^D(N)/\langle w_m \rangle)$. In the cases we are left to study, if we know that $X_0^D(N)/\langle w_m \rangle$ is an elliptic curve over $\Q$, we are able to use Magma to compute the rank of $\Jac(X_0^D(N)/\langle w_m \rangle)$ over $\mathbb{Q}$. Code for these computations is located in the file \texttt{genus{\_}1{\_}quotients{\_}and{\_}ranks.m} in \cite{Rep}. 

In summary, we determine all $41$ triples $(D,N,m)$ with relatively prime $D,N > 1$  and $N$ squarefree such that $X_0^D(N)/\langle w_m \rangle$ has genus one. We then work to determine, for each triple, whether the corresponding quotient is an elliptic curve over $\mathbb{Q}$: we use CM points or the results of \cite{NR15},\cite{PS23} to answer this in the affirmative, and use the results on local points from \S 3 or the results of \cite{NR15} to answer this in the negative. When this answer is positive, we use Ribet's isogeny to determine ranks over $\mathbb{Q}$. 

Our results are summarized in \cref{table: squarefree}. Reasonings for whether the quotient is rationally bielliptic provided in the table should not be taken to be exhaustive; in some places we provide multiple reasonings, but in general arguments other than those listed may be successful. The only triples for which we are unable to determine whether the quotient has a rational point are $(6,13,6),(6,23,69),(21,5,21),(26,5,26),(35,3,35),$ and $(38,3,38)$, but we still determine that the rank of $\text{Jac}\left(X_0^6(23)/\langle w_{69} \rangle\right)$ over $\mathbb{Q}$ is $0$ in each of these cases.

\rowcolors{2}{white}{gray!20}
\begin{longtable}{|l|l|l|l|}
\caption{Squarefree $N$ and bielliptic Atkin--Lehner involutions}\label{table: squarefree} \\ \hline 
\rowcolor{gray!50}
$(D,N,m)$ & $g(X_0^D(N))$ & $X_0^D(N)/\langle w_m \rangle (\mathbb{Q}) \neq \emptyset$? & reason \\ \hline
$(6,5,3)$  & $1$  &  no & \cite{NR15} \\ \hline
$(6,5,5)$  & $1$  &  no & \cite{NR15}  \\ \hline
$(6,5,15)$  & $1$ & yes (rank $0$) & \cite{NR15}, CM-points  \\ \hline
$(6,7,2)$  & $1$  & no & \cite{NR15},\cite{Ogg83}  \\ \hline
$(6,7,7)$  & $1$  & yes (rank $0$) & \cite{NR15}  \\ \hline
$(6,7,14)$  & $1$  & yes (rank $0$) & \cite{NR15}, CM-points  \\ \hline
$(6,11,6)$  & $3$   & no & \cite{NR15},\cite{PS23}  \\ \hline
$(6,11,22)$  & $3$   & no & \cite{NR15},\cite{PS23}  \\ \hline
$(6,11,33)$  & $3$   & no & \cite{NR15},\cite{PS23}  \\ \hline
$(6,13,6)$  & $1$   & \textbf{not known} (rank 0) & N/A  \\ \hline
$(6,13,26)$  & $1$   & yes (rank $0$) & CM-points  \\ \hline
$(6,13,39)$  & $1$   & yes (rank $0$) & CM-points  \\ \hline
$(6,17,2)$  & $3$   & no & \cite{PS23}  \\ \hline
$(6,17,51)$  & $3$   & yes (rank $0$) & \cite{PS23}  \\ \hline
$(6,17,102)$  & $3$   & yes (rank $1$) & \cite{PS23}  \\ \hline
$(6,19,3)$  & $3$   & no & \cite{PS23}  \\ \hline
$(6,19,19)$  & $3$   & no & \cite{PS23} \\ \hline
$(6,19,57)$  & $3$   & no & \cite{PS23}  \\ \hline
$(6,23,46)$  & $5$   & no & \cite{Ogg85} \\ \hline
$(6,23,69)$  & $5$  & \textbf{not known} (rank $0$) & N/A   \\ \hline
$(6,23,138)$  & $5$  & yes (rank $1$) & CM-points  \\ \hline
$(6,41,246)$  & $7$  & yes (rank $1$) & CM-points  \\ \hline
$(6,43,129)$  & $7$  & no & \cite{Ogg83}  \\ \hline
$(6,47,94)$  & $9$   & no & \cite{Ogg85}  \\ \hline
$(6,71,426)$  & $13$  & yes (rank $1$) & CM-points  \\ \hline
$(10,3,6)$  & $1$   & no & \cite{NR15}  \\ \hline
$(10,3,10)$  & $1$   & yes (rank $0$) & \cite{NR15}, CM-points \\ \hline
$(10,3,15)$  & $1$   & yes (rank $0$) & \cite{NR15}, CM-points \\ \hline
$(10,7,2)$  & $1$   & no & \cite{NR15},\cite{Ogg85}  \\ \hline
$(10,7,7)$  & $1$   & no & \cite{NR15}  \\ \hline
$(10,7,14)$  & $1$   & no & \cite{NR15}  \\ \hline
$(10,13,10)$  & $3$   & no & \cite{PS23}  \\ \hline
$(10,13,13)$  & $3$   & no & \cite{PS23}  \\ \hline
$(10,13,130)$  & $3$  & yes (rank $1$) & \cite{PS23}, CM-points  \\ \hline
$(10,17,170)$  & $7$   & yes (rank $1$) & CM-points \\ \hline
$(10,29,290)$  & $11$   & yes (rank $1$) & CM-points  \\ \hline
$(10,31,62)$  & $9$   & no & \cite{Ogg85}  \\ \hline
$(14,3,2)$  & $3$   & no & \cite{PS23}  \\ \hline
$(14,3,21)$  & $3$   & yes (rank $0$) & \cite{PS23}, CM-points   \\ \hline
$(14,3,42)$  & $3$   & yes (rank $0$) &  \cite{PS23}, CM-points \\ \hline
$(14,5,2)$  & $3$   & no & \cite{PS23} \\ \hline
$(14,5,35)$ \footnote{There is a typo in \cite{GY17}. The formula of $w_{35} \in W_0(14,5)$ has an extra minus sign in the second coordinate.}  & $3$   & yes (rank $0$) &  \cite{PS23}, CM-points  \\ \hline
$(14,5,70)$  & $3$   & yes (rank $0$) &  \cite{PS23}, CM-points  \\ \hline
$(14,13,182)$  & $7$   & yes (rank $0$) & CM-points  \\ \hline
$(14,19,266)$  & $11$   & yes (rank $0$) & CM-points  \\ \hline
$(15,2,3)$  & $3$   & yes (rank $0$) & \cite{PS23}  \\ \hline
$(15,2,10)$  & $3$   & no & \cite{PS23}  \\ \hline
$(15,2,30)$ & $3$   & yes (rank $0$) & \cite{PS23}, CM-points  \\ \hline
$(15,7,3)$  & $5$   & no & \cite{Ogg85}  \\ \hline
$(15,7,7)$  & $5$   & no & \cite{Ogg85} \\ \hline
$(15,7,105)$  & $5$   & yes (rank $0$) & CM-points \\ \hline
$(15,11,55)$  & $9$   & no & \cite{Ogg85}  \\ \hline
$(15,13,195)$  & $9$   & yes (rank $0$) & CM-points  \\ \hline
$(15,17,255)$  & $13$   & yes (rank $0$) & CM-points  \\ \hline
$(21,2,2)$  & $3$   & no &  \cite{PS23}  \\ \hline
$(21,2,21)$  & $3$   & yes (rank $0$) & \cite{PS23}  \\ \hline
$(21,2,42)$  & $3$   & yes (rank $0$) & \cite{PS23}  \\ \hline
$(21,5,15)$  & $5$   & no & \cite{Ogg85}  \\ \hline
$(21,5,21)$  & $5$   & \textbf{not known} (rank $0$) & N/A\\ \hline
$(21,5,105)$  & $5$  & yes (rank $0$) & CM-points  \\ \hline
$(21,11,231)$  & $13$  & yes (rank $0$) & CM-points  \\ \hline
$(22,3,3)$  & $3$   & no & \cite{NR15},\cite{PS23}  \\ \hline
$(22,3,11)$  & $3$  & no & \cite{NR15},\cite{PS23}  \\ \hline
$(22,3,33)$  & $3$  & no & \cite{NR15},\cite{PS23} \\ \hline
$(22,7,14)$  & $5$  & no & \cite{NR15}  \\ \hline
$(22,7,77)$  & $5$  & yes (rank $0$) & \cite{NR15} \\ \hline
$(22,7,154)$  & $5$  & yes (rank $1$) & CM-points  \\ \hline
$(22,17,374)$  & $15$  & yes (rank $1$) & CM-points \\ \hline
$(26,5,26)$  & $7$  & \textbf{not known} (rank $0$) & N/A  \\ \hline
$(33,2,66)$  & $5$  & yes (rank $0$) & CM-points  \\ \hline
$(33,5,55)$  & $9$  & no & \cite{Ogg85} \\ \hline
$(33,7,231)$  & $13$ & yes (rank $0$) & CM-points  \\ \hline
$(34,3,17)$  & $5$   & no & \cite{Ogg83}  \\ \hline
$(35,2,35)$  & $7$   & yes (rank $0$) & CM-points \\ \hline
$(35,3,35)$  & $9$   & \textbf{not known} (rank $0$)  &  N/A \\ \hline
$(38,3,38)$  & $7$   & \textbf{not known} (rank $0$)  & N/A  \\ \hline
$(46,5,230)$  & $11$  & yes (rank $0$) & CM-points \\ \hline
\end{longtable}

\subsection{Non-squarefree level $N$}\label{not_sqfree_section}
We have $56$ candidate pairs $(D,N)$ where $N$ is not squarefree, listed in \texttt{not{\_}sqfree{\_}candidate{\_}pairs.m}. For each of these pairs except for
\[
 (6, 25),(10, 9),(14, 9),(15, 8),(21, 4),(22, 9),(33, 4),(39,4),
\]
we compute using the code for genera and fixed point counts in \texttt{quot{\_}genus.m} that there exists $m \parallel DN$ such that 
\[
\#X_0^D(N)^{w_m} \ne 2g(X_0^D(N))-2, \quad \text{and} \quad  \#X_0^D(N)^{w_m}> 8,
\]
thus, by \cref{lemma: more-than-8}, $X_0^D(N)$ is not bielliptic.

For $(D,N) \in \{ (14,9),(21,4),(22,9),(33,4) \}$ we have $g \in \{ 7, 11 \},$ and using \cref{not-div} we obtain that $X_0^{D}(N)$ can only have a bielliptic involution of Atkin--Lehner type.

For $(D,N) \in \{ (15,8), (21,4), (39,4) \}$ we have that $X_0^D(N)$ has a bielliptic Atkin--Lehner involution. As $g(X_0^D(N)) > 6$, it follows that this bielliptic involution is unique. 

For $(D,N) \in \{ (6,25),(10,9) \}$ the curve $X_0^D(N)$ has genus $5$. While this curve has exactly one bielliptic involution of Atkin--Lehner type, we are not sure whether it also admits a bielliptic involution that is not of Atkin--Lehner type (see \cref{genus_5_not_squarefree_rmk}). We compute using Magma that $J_0(D\cdot N)^{D-\new}$ contains no positive rank elliptic curves, though, so by Ribet's isogeny (see \cref{ribet_general_level_rmk}) we know that $X_0^D(N)$ has finitely many quadratic points.

\rowcolors{2}{white}{gray!20}
\begin{longtable}{|l|l|l|l|}
\caption{Non-squarefree $N$ and bielliptic Atkin--Lehner involutions}\label{table: non-squarefree} \\ \hline 
\rowcolor{gray!50}
$(D,N,m)$ & $g(X_0(D,N))$ & $X_0^D(N)/\langle w_m \rangle (\mathbb{Q}) \neq \emptyset$? & reason \\ \hline
$(6,25,150)$  & $5$  & \textbf{not known} (rank $0$) & N/A   \\ \hline
$(10,9,90)$  & $5$  &  yes (rank $0$) & \cite{NR15}  \\ \hline
$(15,8,15)$  & $9$  &  no & \cite{NR15} and \cite{Ogg85} \\ \hline
$(21,4,7)$  & $7$  &  no & \cite{NR15} and \cite{Ogg85} \\ \hline
$(39,4,39)$  & $13$  &  no & \cite{NR15} and \cite{Ogg85} \\ \hline
\end{longtable}

\section{Sporadic points on $X_0^D(N)$}
We begin by recalling the definition of a sporadic point:

\begin{definition}
Let $X$ be a curve over a number field $F$. A point $x \in X$ is \textbf{sporadic} if $\text{deg}(x) := [F(x) : F] < \textnormal{a.irr}_F(X)$. In other words, $x$ is sporadic if there are only finitely many points $y \in X$ with $\text{deg}(y) \leq \text{deg}(x)$.
\end{definition}

In \cite[\S 10]{Saia24}, the author pursued the question of whether the curves $X_0^D(N)$ and $X_1^D(N)$ with $D>1$ and $\gcd(D,N) = 1$ have a sporadic point, following the pursuit of the same question in the $D=1$ case in \cite[\S 8]{CGPS22}. For both of these families of curves, it is proven using a combination of \cref{theorem: Abr} and a result of Frey \cite{Frey94}, which bounds $\textnormal{a.irr}_\mathbb{Q}(X_0^D(N))$ above and below in terms of $\text{gon}_\mathbb{Q}(X_0^D(N))$, that $X_0^D(N)$ and $X_1^D(N)$ have sporadic CM points for $DN$ sufficiently large. 

This work then narrows down the list of pairs $(D,N)$ for which it remains to be proven whether $X_0^D(N)$ has a sporadic point (and similarly for $X_1^D(N)$), by using known results on $\textnormal{a.irr}_{\mathbb{Q}}(X_0^D(N))$ and $\textnormal{a.irr}_{\mathbb{Q}}(X_1^D(N))$ and computations of the least degrees of CM points on these curves in \cite{SaiaRep}. In the case of $D>1$, the following main result is reached:

\begin{theorem}{\cite[Thm 10.9]{Saia24}}\label{sp_thm_Saia}
\begin{enumerate}
\item For all but at most $393$ explicit pairs $(D,N)$, consisting of a rational quaternion discriminant $D>1$ and a positive integer $N$ coprime to $D$, the Shimura curve $X_0^D(N)$ has a sporadic CM point. For at least $64$ of these pairs, $X_0^D(N)$ has no sporadic points. 
\item For all but at most $394$ explicit pairs $(D,N)$, consisting of a rational quaternion discriminant $D>1$ and a positive integer $N$ coprime to $D$, the Shimura curve $X_1^D(N)$ has a sporadic CM point. For at least $54$ of these pairs, $X_1^D(N)$ has no sporadic points.  
\end{enumerate}
\end{theorem}
In this section, we apply our results on bielliptic Shimura curves $X_0^D(N)$ to improve on \cref{sp_thm_Saia}.

\subsection{Shimura curves with infinitely many quadratic points}

The result of Hindry \cite[Remarque, p. 221]{Hindry} and Harris--Silverman \cite[Corollary 3]{HaSi91} mentioned in the introduction states that if $X$ is a curve of genus at least $2$ over a number field $F$ and has $\textnormal{a.irr}_F(X) = 2$, then $X$ is either hyperelliptic or is bielliptic with a degree $2$ map to an elliptic curve over $F$ of positive rank. We recalled in \S 1 the full list of hyperelliptic curves $X_0^D(N)$ of genus at least $2$, and the list of $D$ such that $X_0^D(1)$ is not hyperelliptic and has infinitely many quadratic points. Combining this with our study of bielliptic curves $X_0^D(N)$, we obtain the following result.

\begin{theorem}\label{delta_eq_2}
We have that $\textnormal{a.irr}_{\mathbb{Q}}(X_0^D(N)) = 2$ with $D>1$ and $\gcd(D,N) = 1$ if and only if the pair $(D,N)$ is in the following set:
\begin{align*} 
\{
 &(6,1),(6,5), (6,7),(6,11), (6,13), (6,17),(6,19),(6,23),(6,29),(6,31),(6,37), \\
 &(6,41),(6,71),(10,1),(10,3),(10,7), (10,11), (10,13),(10,17),(10,23),(10,29), \\
 &(14,1),(14,5),(15,1),(15,2),(21,1),(22,1),(22,3),(22,5),(22,7),(22,17), \\
 & (26,1),(33,1),(34,1),(35,1), (38,1),(39,1),(39,2), (46,1),(51,1),(55,1), \\
 & (57,1),(58,1),(62,1),(65,1),(69,1),(74,1),(77,1),(82,1),(86,1),(87,1)\\
& (94,1),(95,1),(106,1),(111,1),(118,1),(119,1),(122,1),(129,1),(134,1), \\
&(143,1),(146,1),(159,1),(166,1),(194,1),(206,1),(210,1),(215,1),(314,1),\\
& (330,1),(390,1),(510,1),(546,1)\}. 
\end{align*}
\end{theorem}
\begin{proof}
    This follows immediately from our work in \cref{sqfree_section} and \cref{not_sqfree_section}, determining which bielliptic Shimura curves $X_0^D(N)$ with $N>1$ admit a degree $2$ map over $\Q$ to an elliptic curve of positive rank over $\Q$, with relevant computations being performed in \texttt{genus{\_}1{\_}quotients{\_}and{\_}ranks.m} in \cite{Rep}. 
\end{proof}

\subsection{Sporadic points}
In \cite[\S 10]{Saia24}, the author describes and implements computations of the least degree of a CM point on $X_0^D(N)$ or $X_1^D(N)$ for a fixed pair $(D,N)$ with $D>1$ and $\gcd(D,N) = 1$. We denote these quantities by $d_{\text{CM}}(X_0^D(N))$ and $d_{\text{CM}}(X_1^D(N))$ for the respective curves. Using such computations and \cref{delta_eq_2}, we arrive at the following theorem. 

\begin{theorem}\label{sporadic_thm}
    \begin{enumerate}
        \item For the pairs $(D,N)$ in the following list, the curve $X_0^D(N)$ has no sporadic points:
        \[\{ (6,17),(6,23),(6,41),(6,71),(10,13),(10,17),(10,29),(22,7),(22,17) \}. \]
        \item For $(D,N) \in \{(6,23), (6, 71),(10, 17),(10, 29)\}$, the curve $X_1^D(N)$ has no sporadic CM points.
        \item Of the $320$ pairs for which we remain unsure of whether $X_0^D(N)$ has a sporadic CM point following \cref{sp_thm_Saia} and part (1), all but at most $56$ have a sporadic CM point. We list these $56$ pairs in \cref{unknown_X0_table}.
        \item Of the $336$ pairs for which we remain unsure of whether $X_1^D(N)$ has a sporadic CM point following \cref{sp_thm_Saia} and part (2), all but at most $263$ have a sporadic CM point. These pairs comprise the union of those in \cref{unknown_X0_table} and those in \cref{unknown_X1_table}. 
    \end{enumerate}
\end{theorem}
\begin{proof}
    \begin{enumerate}
        \item For these pairs we have proven that $\textnormal{a.irr}_\mathbb{Q}(X_0^D(N)) = 2$ by virtue of having a bielliptic quotient with positive rank over $\mathbb{Q}$. We know that $X_0^D(N)(\mathbb{Q}) = \emptyset$ for $D>1$, so these curves cannot have sporadic points. 
        \item For these pairs, we have $\text{a.irr}(X_0^D(N)) = 2$, giving the inequality
        \[ \text{a.irr}(X_1^D(N)) \leq 2 \cdot \text{deg}\left(X_1^D(N) 
        \rightarrow X_0^D(N)\right) = 2 \cdot \text{max}\{1,\phi(N)/2)\}. \]
        We compute for each that
        \[ \text{max}\{2,\phi(N)\} \leq d_\textnormal{CM}(X_1^D(N)), \]
        and thus there are no sporadic CM points. 
        \item For these $320$ unknown pairs, we compute that $d_{\text{CM}}(X_0^D(N)) = 2$ for all but $56$. For all of these pairs, the curve $X_0^D(N)$ is not among those listed in \cref{delta_eq_2}, and thus any CM point of degree $2$ is sporadic.
        \item If $2 < \textnormal{a.irr}(X_0^D(N)) \leq \textnormal{a.irr}(X_1^D(N))$ and $d_{\textnormal{CM}}(X_1^D(N)) = 2$, then each quadratic CM point on $X_1^D(N)$ is sporadic. Of the remaining $336$ pairs, there are $73$ that satisfy these conditions and thus have a sporadic CM point.  
    \end{enumerate}
    Code for all computations mentioned above can be found in \texttt{narrow{\_}sporadics.m} in \cite{Rep}. 
\end{proof}

\begin{longtable}{| c | c | c | c || c | c | c | c |}\caption{The $56$ pairs $(D,N)$ with $D>1$ and $\gcd(D,N) = 1$ for which we remain unsure of whether $X_0^D(N)$ has a sporadic point}\label{unknown_X0_table} \\ \hline 
$D$ & $N$ & $g(X_0^D(N))$ & $d_\textnormal{CM}(X_0^D(N))$ & $D$ & $N$ & $g(X_0^D(N))$ & $d_\textnormal{CM}(X_0^D(N))$ \\ \hline \hline

$ 6 $
& $ 155 $ & $ 33 $ & $ 4 $ & $ 51 $ & $ 5 $ & $ 17 $ & $ 4 $ \\ \hline
& $ 203 $ & $ 41 $ & $ 4 $ & & $ 10 $ & $ 49 $ & $ 4 $ \\ \hline
& $ 287 $ & $ 57 $ & $ 4 $ & & $ 20 $ & $ 97 $ & $ 6 $ \\ \hline
& $ 295 $ & $ 61 $ & $ 4 $ & $ 55 $ & $ 8 $ & $ 41 $ & $ 4 $ \\ \hline
& $ 319 $ & $ 61 $ & $ 4 $ & $ 62 $ & $ 15 $ & $ 61 $ & $ 4 $ \\ \hline
$ 10 $
& $ 69 $ & $ 33 $ & $ 4 $ & $69$ & $ 11 $ & $ 45 $ & $ 4 $ \\ \hline
& $ 77 $ & $ 33 $ & $ 4 $ & $ 77 $ & $ 6 $ & $ 61 $ & $ 4 $ \\ \hline
& $ 87 $ & $ 41 $ & $ 4 $ & $ 86 $ & $ 7 $ & $ 29 $ & $ 4 $ \\ \hline
& $ 119 $ & $ 49 $ & $ 4 $ & $ 87 $ & $ 8 $ & $ 57 $ & $ 4 $ \\ \hline
& $ 141 $ & $ 65 $ & $ 4 $ & $ 95 $ & $ 3 $ & $ 25 $ & $ 4 $ \\ \hline
& $ 161 $ & $ 65 $ & $ 4 $ & $ 111 $ & $ 2 $ & $ 19 $ & $ 4 $ \\ \hline
& $ 191 $ & $ 65 $ & $ 4 $ & & $ 4 $ & $ 37 $ & $ 4 $ \\ \hline
$ 14 $
& $ 39 $ & $ 29 $ & $ 4 $ & $ 119 $ & $ 6 $ & $ 97 $ & $ 6 $ \\ \hline
& $ 87 $ & $ 61 $ & $ 4 $ & $ 122 $ & $ 7 $ & $ 41 $ & $ 4 $ \\ \hline
& $ 95 $ & $ 61 $ & $ 4 $ & $ 129 $ & $ 7 $ & $ 57 $ & $ 4 $ \\ \hline
$ 15 $
& $ 34 $ & $ 37 $ & $ 4 $ & $ 134 $ & $ 3 $ & $ 23 $ & $ 4 $ \\ \hline
$ 21 $
& $ 38 $ & $ 61 $ & $ 4 $ & & $ 9 $ & $ 67 $ & $ 4 $ \\ \hline
$ 22 $
& $ 35 $ & $ 41 $ & $ 4 $ & $ 143 $ & $ 2 $ & $ 31 $ & $ 4 $ \\ \hline
& $ 51 $ & $ 61 $ & $ 4 $ & & $ 4 $ & $ 61 $ & $ 4 $ \\ \hline
$ 26 $
& $ 21 $ & $ 33 $ & $ 4 $ & $ 146 $ & $ 7 $ & $ 49 $ & $ 4 $ \\ \hline
$ 33 $
& $ 16 $ & $ 41 $ & $ 4 $ & $ 183 $ & $ 5 $ & $ 61 $ & $ 4 $ \\ \hline
$ 34 $
& $ 29 $ & $ 41 $ & $ 4 $ & $ 194 $ & $ 3 $ & $ 33 $ & $ 4 $ \\ \hline
& $ 35 $ & $ 65 $ & $ 4 $ & $ 215 $ & $ 2 $ & $ 43 $ & $ 4 $ \\ \hline
$ 35 $
& $ 12 $ & $ 49 $ & $ 4 $ & & $ 3 $ & $ 57 $ & $ 4 $ \\ \hline
$ 38 $
& $ 21 $ & $ 49 $ & $ 4 $ & $ 326 $ & $ 3 $ & $ 55 $ & $ 4 $ \\ \hline
$ 39 $
& $ 10 $ & $ 37 $ & $ 4 $ & $ 327 $ & $ 2 $ & $ 55 $ & $ 4 $ \\ \hline
& $ 31 $ & $ 65 $ & $ 4 $ & $ 335 $ & $ 2 $ & $ 67 $ & $ 4 $ \\ \hline
$ 46 $
& $ 15 $ & $ 45 $ & $ 4 $ & $ 390 $ & $ 7 $ & $ 65 $ & $ 4 $ \\ \hline

\end{longtable}

\begin{longtable}{|c|c|c|c|c|c|c|c|}\caption{The $207$ pairs $(D,N)$ with $D>1$ and $\text{gcd}(D,N) = 1$ which are not included in \cref{unknown_X0_table} for which we remain unsure of whether $X_1^D(N)$ has a sporadic point}\label{unknown_X1_table} \\ \hline 
$ ( 6 , 5 ) $ & 
$ ( 6 , 7 ) $ & 
$ ( 6 , 13 ) $ & 
$ ( 6 , 17 ) $ & 
$ ( 6 , 19 ) $ & 
$ ( 6 , 25 ) $ & 
$ ( 6 , 29 ) $ & 
$ ( 6 , 31 ) $ \\ \hline
$ ( 6 , 35 ) $ & 
$ ( 6 , 37 ) $ & 
$ ( 6 , 41 ) $ & 
$ ( 6 , 43 ) $ & 
$ ( 6 , 47 ) $ & 
$ ( 6 , 49 ) $ & 
$ ( 6 , 53 ) $ & 
$ ( 6 , 55 ) $ \\ \hline
$ ( 6 , 59 ) $ & 
$ ( 6 , 61 ) $ & 
$ ( 6 , 65 ) $ & 
$ ( 6 , 67 ) $ & 
$ ( 6 , 73 ) $ & 
$ ( 6 , 77 ) $ & 
$ ( 6 , 79 ) $ & 
$ ( 6 , 83 ) $ \\ \hline
$ ( 6 , 85 ) $ & 
$ ( 6 , 89 ) $ & 
$ ( 6 , 91 ) $ & 
$ ( 6 , 95 ) $ & 
$ ( 6 , 97 ) $ & 
$ ( 6 , 101 ) $ & 
$ ( 6 , 103 ) $ & 
$ ( 6 , 107 ) $ \\ \hline
$ ( 6 , 109 ) $ & 
$ ( 6 , 113 ) $ & 
$ ( 6 , 115 ) $ & 
$ ( 6 , 119 ) $ & 
$ ( 6 , 121 ) $ & 
$ ( 6 , 125 ) $ & 
$ ( 6 , 127 ) $ & 
$ ( 6 , 131 ) $ \\ \hline
$ ( 6 , 133 ) $ & 
$ ( 6 , 137 ) $ & 
$ ( 6 , 139 ) $ & 
$ ( 6 , 143 ) $ & 
$ ( 6 , 145 ) $ & 
$ ( 6 , 149 ) $ & 
$ ( 6 , 151 ) $ & 
$ ( 6 , 157 ) $ \\ \hline
$ ( 6 , 161 ) $ & 
$ ( 6 , 163 ) $ & 
$ ( 6 , 167 ) $ & 
$ ( 6 , 169 ) $ & 
$ ( 6 , 173 ) $ & 
$ ( 6 , 179 ) $ & 
$ ( 6 , 181 ) $ & 
$ ( 6 , 191 ) $ \\ \hline
$ ( 6 , 193 ) $ & 
$ ( 6 , 197 ) $ & 
$ ( 6 , 199 ) $ & 
$ ( 10 , 7 ) $ & 
$ ( 10 , 9 ) $ & 
$ ( 10 , 13 ) $ & 
$ ( 10 , 19 ) $ & 
$ ( 10 , 21 ) $ \\ \hline
$ ( 10 , 27 ) $ & 
$ ( 10 , 31 ) $ & 
$ ( 10 , 33 ) $ & 
$ ( 10 , 37 ) $ & 
$ ( 10 , 39 ) $ & 
$ ( 10 , 41 ) $ & 
$ ( 10 , 43 ) $ & 
$ ( 10 , 47 ) $ \\ \hline
$ ( 10 , 49 ) $ & 
$ ( 10 , 51 ) $ & 
$ ( 10 , 53 ) $ & 
$ ( 10 , 57 ) $ & 
$ ( 10 , 59 ) $ & 
$ ( 10 , 61 ) $ & 
$ ( 10 , 63 ) $ & 
$ ( 10 , 67 ) $ \\ \hline
$ ( 10 , 71 ) $ & 
$ ( 10 , 73 ) $ & 
$ ( 10 , 79 ) $ & 
$ ( 10 , 83 ) $ & 
$ ( 10 , 89 ) $ & 
$ ( 10 , 91 ) $ & 
$ ( 10 , 97 ) $ & 
$ ( 10 , 103 ) $ \\ \hline
$ ( 14 , 5 ) $ & 
$ ( 14 , 9 ) $ & 
$ ( 14 , 11 ) $ & 
$ ( 14 , 13 ) $ & 
$ ( 14 , 15 ) $ & 
$ ( 14 , 17 ) $ & 
$ ( 14 , 19 ) $ & 
$ ( 14 , 23 ) $ \\ \hline
$ ( 14 , 25 ) $ & 
$ ( 14 , 27 ) $ & 
$ ( 14 , 29 ) $ & 
$ ( 14 , 31 ) $ & 
$ ( 14 , 33 ) $ & 
$ ( 14 , 37 ) $ & 
$ ( 14 , 41 ) $ & 
$ ( 14 , 43 ) $ \\ \hline
$ ( 14 , 47 ) $ & 
$ ( 14 , 53 ) $ & 
$ ( 14 , 59 ) $ & 
$ ( 14 , 61 ) $ & 
$ ( 15 , 8 ) $ & 
$ ( 15 , 11 ) $ & 
$ ( 15 , 13 ) $ & 
$ ( 15 , 14 ) $ \\ \hline
$ ( 15 , 16 ) $ & 
$ ( 15 , 17 ) $ & 
$ ( 15 , 19 ) $ & 
$ ( 15 , 22 ) $ & 
$ ( 15 , 23 ) $ & 
$ ( 15 , 26 ) $ & 
$ ( 15 , 28 ) $ & 
$ ( 15 , 29 ) $ \\ \hline
$ ( 15 , 31 ) $ & 
$ ( 15 , 32 ) $ & 
$ ( 15 , 37 ) $ & 
$ ( 15 , 41 ) $ & 
$ ( 15 , 43 ) $ & 
$ ( 15 , 47 ) $ & 
$ ( 21 , 8 ) $ & 
$ ( 21 , 11 ) $ \\ \hline
$ ( 21 , 13 ) $ & 
$ ( 21 , 16 ) $ & 
$ ( 21 , 17 ) $ & 
$ ( 21 , 19 ) $ & 
$ ( 21 , 23 ) $ & 
$ ( 21 , 25 ) $ & 
$ ( 21 , 29 ) $ & 
$ ( 21 , 31 ) $ \\ \hline
$ ( 22 , 5 ) $ & 
$ ( 22 , 7 ) $ & 
$ ( 22 , 9 ) $ & 
$ ( 22 , 13 ) $ & 
$ ( 22 , 15 ) $ & 
$ ( 22 , 17 ) $ & 
$ ( 22 , 19 ) $ & 
$ ( 22 , 21 ) $ \\ \hline
$ ( 22 , 23 ) $ & 
$ ( 22 , 25 ) $ & 
$ ( 22 , 27 ) $ & 
$ ( 22 , 29 ) $ & 
$ ( 22 , 31 ) $ & 
$ ( 22 , 37 ) $ & 
$ ( 26 , 5 ) $ & 
$ ( 26 , 7 ) $ \\ \hline
$ ( 26 , 9 ) $ & 
$ ( 26 , 11 ) $ & 
$ ( 26 , 15 ) $ & 
$ ( 26 , 17 ) $ & 
$ ( 26 , 19 ) $ & 
$ ( 26 , 23 ) $ & 
$ ( 26 , 25 ) $ & 
$ ( 26 , 29 ) $ \\ \hline
$ ( 26 , 31 ) $ & 
$ ( 33 , 8 ) $ & 
$ ( 33 , 13 ) $ & 
$ ( 33 , 17 ) $ & 
$ ( 33 , 19 ) $ & 
$ ( 34 , 5 ) $ & 
$ ( 34 , 9 ) $ & 
$ ( 34 , 11 ) $ \\ \hline
$ ( 34 , 13 ) $ & 
$ ( 34 , 15 ) $ & 
$ ( 34 , 19 ) $ & 
$ ( 34 , 23 ) $ & 
$ ( 35 , 8 ) $ & 
$ ( 35 , 9 ) $ & 
$ ( 35 , 11 ) $ & 
$ ( 35 , 13 ) $ \\ \hline
$ ( 38 , 7 ) $ & 
$ ( 38 , 9 ) $ & 
$ ( 38 , 11 ) $ & 
$ ( 38 , 13 ) $ & 
$ ( 38 , 17 ) $ & 
$ ( 39 , 5 ) $ & 
$ ( 39 , 7 ) $ & 
$ ( 39 , 8 ) $ \\ \hline
$ ( 39 , 11 ) $ & 
$ ( 46 , 9 ) $ & 
$ ( 46 , 11 ) $ & 
$ ( 46 , 13 ) $ & 
$ ( 46 , 17 ) $ & 
$ ( 51 , 8 ) $ & 
$ ( 51 , 11 ) $ & 
$ ( 57 , 7 ) $ \\ \hline
$ ( 58 , 5 ) $ & 
$ ( 58 , 9 ) $ & 
$ ( 58 , 11 ) $ & 
$ ( 58 , 13 ) $ & 
$ ( 62 , 7 ) $ & 
$ ( 62 , 9 ) $ & 
$ ( 62 , 11 ) $ & 
$ ( 65 , 7 ) $ \\ \hline
$ ( 74 , 5 ) $ & 
$ ( 74 , 7 ) $ & 
$ ( 82 , 5 ) $ & 
$ ( 87 , 5 ) $ & 
$ ( 91 , 5 ) $ & 
$ ( 106 , 5 ) $ & 
$ ( 122 , 5 ) $ & \\ \hline 

\end{longtable}

\section{There are no geometrically trigonal Shimura curves $X_0^D(N)$}\label{trigonal_section}

Shimura curves with $D>1$ have no real points. In particular, they have no odd-degree points. If $X$ is such a curve and $X$ has a degree $d$ map to either $\mathbb{P}^1_{\mathbb{Q}}$ or an elliptic curve $E$ over $\mathbb{Q}$ (moreover, over $\mathbb{R}$), then it follows that $d$ is even. On the other hand, this does not preclude the existence of Shimura curves which have odd geometric gonality, or which admit odd-degree maps to elliptic curves over $\Qbar$. 

There do indeed exist geometrically trielliptic Shimura curves. For example, if $D>1$ is odd and $X_0^D(1)$ has genus $1$ (so, is a pointless genus $1$ curve over $\mathbb{Q}$), then $X_0^D(2)$ is a degree $3$ cover of $X_0^D(1)$ and hence is trielliptic over a degree $2$ extension. This applies to $D \in \{15,21,33\}$. Similarly, $X_0^{10}(9)$ is geometrically trielliptic with a degree $3$ map to $X_0^{10}(3)$. (It immediately follows from these examples that $\text{a.irr}_{\mathbb{Q}}(X_0^{D}(N)) \leq 6$ for $(D,N) \in \{(10,9),(21,2),(33,2)\}$. However, we know from \cref{table: non-squarefree} and \cref{table: squarefree} that each of these three pairs is also \emph{bielliptic} over $\Q$, and so in fact we have $\text{a.irr}_{\mathbb{Q}}(X_0^{D}(N)) = 4$ for $(D,N) \in \{(10,9),(21,2),(33,2)\}$.)

The question of whether the above examples, coming from natural modular maps, are the only geometrically trielliptic Shimura curves in this family will be relegated to future work. On the other hand, we \emph{will} determine here all of the geometrically trigonal Shimura curves $X_0^D(N)$ with $\gcd(D,N) = 1$.    

\begin{definition}\label{trigonal_def}
A curve $X$ of genus $g \ge 2$ over a number field $F$ is \textbf{trigonal (over $F$)} if there is a degree $3$ finite map $X \rightarrow \mathbb{P}^1_F$. We call $X$ \textbf{geometrically trigonal} if there exists a non-constant morphism 
\[ X \otimes_{\text{Spec} F} \text{Spec} \overline{F} \longrightarrow \P^1_{\overline{F}} \] 
of degree $3$.
\end{definition}


The following two results of Schweizer will be our primary tools towards our main result of this section. 

\begin{lemma}\cite[Lemma 3.4]{Schweizer15} \label{lemma: trigonal}
Let $X$ be a trigonal curve of genus $g$ and $\sigma$ an involution on $X$.
\begin{enumerate}[(a)]
    \item If $g$ is odd, then $\sigma$ has exactly $4$ fixed points.
    \item If $g$ is even, then $\sigma$ has $2$ or $6$ fixed points.
\end{enumerate}
\end{lemma}

\begin{proposition}\cite[Corollary 3.5]{Schweizer15}
    Let $X$ be a curve of genus $g \equiv 1 \;(\bmod \; 4)$. If $\Aut(X)$ has a subgroup $H$ isomorphic to $\Z/2\Z \oplus \Z/2\Z$, then $X$ cannot be trigonal. 
\end{proposition}

Using the above results, we obtain the following. 

\begin{theorem}\label{proposition: trigonal-prop}
    The geometrically trigonal Shimura curves $X_0^D(N)$ with $\gcd(D,N)=1$ and $D>1$ are exactly those with
    \[ (D,N) \in \{(26,1),(38,1),(58,1),(106,1),(118,1)\} . \]
\end{theorem}
\begin{proof}
    If $X_0^D(N)$ is geometrically trigonal, then by \cref{theorem: Abr} we must have $g(X_0^D(N)) \le 29$. There are $455$ relatively prime pairs $(D,N)$ with $D>1$ such that $g(X_0^D(N)) \le 29$. These are computed in \texttt{generating\_trigonal\_candidates.m} and listed in \texttt{trigonal\_candidate\_pairs.m} in \cite{Rep}. In the file \texttt{trigonal\_checks.m}, we compute that the only candidates pairs which have $g(X_0^D(N)) \ge 2$ and satisfy the requirements to be trigonal from both of the above results of Schweizer are 
    \[ (D,N) \in \{(26,1),(38,1),(58,1),(106,1),(118,1),(214,1)\}. \]
    The curves $X_0^D(1)$ for $D \in \{26,38,58\}$ are all genus $2$, and hence are also geometrically trigonal (one can see this from a Riemann--Roch argument, considering the divisor $3[P]$ for any point $P \in X_0^D(1)(\C)$). For $D \in \{106,118\}$, the curve $X_0^D(1)$ is non-hyperelliptic of genus $4$ and hence is geometrically trigonal (one can see this implication, for example, from \cite[Prop. A.1. (v)]{Poonen}). The final candidate curve $X_0^{214}(1)$ is \emph{not} trigonal. We can see this as follows: the quotients $X = X_0^{214}(1)/\langle w_{107} \rangle$ and $Y = X_0^{214}(1)/W_0(214,1)$ are of genus $5$ and genus $1$, respectively. The curve $X$ is geometrically bielliptic, being a degree $2$ cover of $Y$. It must then have geometric gonality exactly $4$ by a Castelnuovo--Severi inequality argument, given that its genus is $5$. As $X_0^{214}(1)$ is a cover of $X$, it too cannot have geometric gonality $3$ by \cite[Prop. A.1 (vii)]{Poonen}. 
\end{proof}

\noindent \textbf{Competing Interest Statement:} The authors have no competing interests to declare that are relevant to the content of this article. \\

\noindent \textbf{Data Availability Statement:} All code used to perform the computations described in this paper, along with data generated from this code, can be found in the public Github repository \cite{Rep}.

\bibliographystyle{amsalpha}
\bibliography{biblio}
\end{document}